\documentclass[12pt]{article}
\usepackage{amsthm,enumerate} 
\usepackage[mathscr]{euscript}
\usepackage{amssymb,amsmath}
\usepackage{hyperref}
\usepackage{endnotes}
\usepackage{graphicx} 
\usepackage{xcolor}
\newtheorem{theorem}{Theorem}

\newtheorem{proposition}[theorem]{Proposition}%
\newtheorem{corollary}[theorem]{Corollary}

\def\ma{\mathfrak{A}}
\def\mb{\mathfrak{B}}
\def\mc{\mathfrak{C}}
\def\mm{\mathfrak{M}}
\def\mn{\mathfrak{N}}
\def\SH{\mbox{SH}}
\def\Th{\mbox{Th}}

\def\dom{\mbox{dom}}

\def\Part{\mbox{\rm Part}}

\begin{document}

\title{On highly equivalent non-isomorphic countable models of arithmetic and set theory\thanks{The second author would like to thank  the Academy of Finland, grant no: 322795, and funding from the European Research Council (ERC) under the
European Union’s Horizon 2020 research and innovation programme (grant agreement No
101020762). }}
\author{Tapani Hyttinen\\
University of Helsinki, Finland
\and 
Jouko V\"a\"an\"anen\\
University of Helsinki, Finland\\
University of Amsterdam, The Netherlands}
\maketitle
\def\PA{\mbox{PA}}
\def\ZFC{\mbox{ZFC}}
\def\EF{\mbox{EF}}

\begin{abstract}
It is well-known that the first order Peano axioms $\PA$ have a continuum of non-isomorphic countable models.
The question, how close to being isomorphic such countable models can be, seems to be less investigated.
A measure of closeness to isomorphism of countable models is the length of back-and-forth sequences that can be established between them. We show that for every countable ordinal $\alpha$ there are countable non-isomorphic models of $\PA$ with a back-and-forth sequence of length $\alpha$ between them. This implies that the Scott height (or rank) of such models is bigger than $\alpha$. We also prove the same result for models of $\ZFC$.
\end{abstract}

It is well-known that there are continuum many non-isomorphic countable models of $\PA$ i.e. of Peano's axioms for the arithmetic of natural numbers, all elementarily equivalent to the standard model. This can be seen as evidence of the richness of the class of countable models of $\PA$. In this paper we prove a different kind of richness result for arithmetic. We show that for every countable $\alpha$ there are countable non-isomorphic models of $\PA$ that are not only elementarily equivalent (to the standard model) but are even $\alpha$-equivalent to each other. This means that the models cannot be distinguished from each other by any sentence of $L_{\infty\omega}$ of quantifier-rank $<\alpha$. So it is not only that there are many countable non-standard models of $\PA$ but there are also many countable models of $\PA$ which are very difficult to distinguish from each other. Our results, based on the methodology of \cite{MR503792}, hold for more general theories than $\PA$, e.g. for $\ZFC$, so this phenomenon is not a question of having missed some axioms from $\PA$. It is feature of all first order theories which satisfy some specific rather weak assumptions, reminiscent of $\PA$. Thus in our result about $\ZFC$ (Theorem~\ref{zfc}) there is nothing specific about $\ZFC$. As the proof shows, the result is more general but we formulate it for $\ZFC$ for simplicity. For a recent results and survey of Scott heights we refer to \cite{MR4402053}.

The so-called \emph{internal categoricity} of $\PA$ (respectively of $\ZFC$) means the following uniqueness fact: If two models of $\PA$, with the same domain but disjoint vocabularies, satisfy the Induction Schema even for formulas where the non-logical symbols from the other model are allowed to occur, the two models are isomorphic (up to a change of vocabulary). We refer to \cite{https://doi.org/10.48550/arxiv.2204.13754} for an overview and history of this concept. In the case that the models are countable and non-isomorphic but $\alpha$-equivalent, we may conclude informally that even $\alpha$-equivalence is not enough to communicate to the models what the other model is like. The question to what extent can two  models of arithmetic (or two people each with their own model of arithmetic) communicate to each other what their model is, was raised in \cite{10.2307/23350653} and discussed extensively in the philosophical literature  afterwords (see \cite{https://doi.org/10.48550/arxiv.2204.13754} for details).

Two models $\ma$ and $\mb$ of the same vocabulary are called \emph{$\alpha$-equivalent} if $\ma\equiv^\alpha_{\infty\omega}\mb$ i.e. if they satisfy the same sentences of $L_{\infty\omega}$ of quantifier-rank $<\alpha$. For any countable models $\ma$ and $\mb$ in a countable vocabulary there is a countable ordinal $\alpha$ such $\alpha$-equivalence of $\ma$ and $\mb$ implies their isomorphism \cite{MR0200133}. Thus countable ordinals provide a perfect method to estimate how close two countable models are to being isomorphic. The bigger the $\alpha$ the closer they are. Every pair of countable models is measured in this way by some $\alpha$.

We discussed above the question how unique are the natural numbers. The concept of $\alpha$-equivalence provides a method to evaluate this question. If we have two countable non-isomorphic but $\alpha$-equivalent models of our theory of the natural numbers, we know that there is nothing we can add to our theory that is expressible in  $L_{\infty\omega}$ by a sentence of quantifier-rank $<\alpha$ that would distinguish the models from each other. Of course, neither of the models can be in such a situation with the standard model because the standard model can be characterized up to isomorphism among models of $\PA$ by means of the sentence
$$\forall x\bigvee_n\exists x_1\ldots x_n\forall y(y<x\to (y=x_1\vee\ldots\vee y=x_n))$$
of quantifier-rank $\omega+1$.

There is a nice characterization of $\alpha$-equivalence due to Carol Karp. To formulate it, let $\Part(\ma,\mb)$ denote the set of all finite partial isomorphisms $\ma\to\mb$. We use $A,B,$ etc to denote the domains of structures $\ma,\mb,$ etc.

\begin{proposition}[\cite{MR0209132}]\label{baf}
$\ma\equiv^\alpha_{\infty\omega}\mb$ if and only if  there are sets $P_\beta$, $\beta\le\alpha$, such that:
\begin{enumerate}
\item $\emptyset \ne P_{\alpha} \subseteq . . . \subseteq P_0 \subseteq \Part(\ma, \mb)$. 
\item $\forall f \in  P_{\beta +1}\forall a \in  A\exists b \in  B\exists g \in  P_\beta (f \cup \{(a, b)\} \subseteq g)$ for $\beta  <{\alpha}$. 
\item $\forall f \in  P_{\beta +1}\forall b \in  B\exists a \in  A\exists g \in  P_\beta (f \cup \{(a, b)\} \subseteq g)$ for $\beta  <{\alpha}$. 

\end{enumerate}We call $(P_\beta)_{\beta\le\alpha}$ a \emph{back-and-forth sequence} for $(\ma,\mb)$.
\end{proposition}

Two countable models $\ma$ and $\mb $ are isomorphic if and only if player II has a winning strategy in the 
Ehrenfeucht-Fra\" \i ss\'e game $\EF_\omega(\ma,\mb)$ of length $\omega$. 
We refer to \cite{MR2768176} for details about $\EF_\omega(\ma,\mb)$ and its transfinite versions.
By
\cite{MR1111753}
there are for any unstable countable $T$, any $\kappa>\omega$ with $\kappa^{<\kappa}=\kappa$, and any $\alpha<\kappa$ models $\ma$ and $\mb$ of $T$ such that $\ma$ and $\mb$ are non-isomorphic but player II has a winning strategy in the transfinite Ehrenfeucht-Fra\" \i ss\'e game of length $\alpha$ between $\ma$ and $\mb$. Even stronger results hold for uncountable models, see \cite{MR1111753} for details. Since the theories we deal with in this paper, e.g. $\PA$ and $\ZFC$, are manifestly unstable,   the existence of non-isomorphic models which are very hard to distinguish from each other (the results of \cite{MR1111753} give much more than $\alpha$-equivalence for all $\alpha<\omega_1$) follows from \cite{MR899084}, if the models are allowed to be uncountable. With reference to this, we emphasise that models in the current paper are countable. Note that the theory $\mbox{DLO}$ of dense linear order without endpoints is unstable but $\aleph_0$-categorical.

Throughout this paper $T$ is a first order theory  in a countable vocabulary containing a binary predicate $\{<\}$. Our main result concerns first order theories $T$ which satisfy the following assumptions:
\begin{description}
\item [A1] The well-ordering schema $$(\forall\vec{y})[(\exists x)\phi(x, \vec{y}) \to (\exists x_0)
(\phi(x_0, \vec{y}) \wedge (\forall x < x_0)\neg\phi(x,\vec{y}))],$$ hence $T$ has definable Skolem-functions.
\item [A2] Every non-minimal element has a predecessor.
\item [A3] There is no largest element. 

\end{description}

Canonical examples of such  a theory is Peano arithmetic $\PA$ and its expansions to richer vocabularies.

For any model $\mm$ of $T$ and $A\subseteq M$ we use $\SH(A)$ to denote the Skolem Hull of $A$ in $\mm$.

\begin{theorem}[\cite{MR503792}]\label{shelah}
If $T$ satisfies A1-A3, then $T$ has a type $p$ such that: 

\begin{enumerate}

\item   
For every $\phi(x, \vec{y})$ there is a finite $p_\phi\subseteq p$ such that if $$\psi^p_\phi(\vec{y}) = (\forall x_1)(\exists x >x_1)[ \bigwedge p_\phi\wedge\phi(x,\vec{y})],$$ then for any model $\mm$ of $T$ and $A\subseteq M$, 
the condition $$\mm\models \psi^p_\phi[\vec{c}]\iff\phi(x,\vec{c}) \in p(A),$$
where  $\vec{c} \in  A$, defines a consistent type $p(A)$ over $A$. Denote $\Psi_p=\{\psi^p_\phi(\vec{y}) : \phi\in L \}$. We say that $\Psi_p$ \emph{defines} $p(A)$.

\item If $\mm\preceq_{end} \mn$, $a\in N\setminus M$ realizes $p(M)$, and $\mn=\SH(M\cup\{a\})$, then $\mn$ is an almost minimal end-extension of $\mm$ i.e. $\mn$ is an end-extension of $\mm$ and there is no $\mn'$ such that $\mm\prec_{end}\mn'\prec_{end} \mn$. (Here $\mm\prec_{end}\mn$ means that $\mn$ is a proper elementary end-extension of $\mm$, and $\mm\preceq_{end}\mn$ means that $\mn$ is just an elementary end-extension of $\mm$.)
\end{enumerate}
\end{theorem}

\begin{theorem}\label{main}Suppose $T$ satisfies A1-A3.
For every $\delta<\omega_1$ there are countable models $\ma$ and $\mb$ of $T$ such that $\ma$ and $\mb$ are $\delta$-equivalent but not isomorphic.
\end{theorem}

\begin{proof}
By Theorem~\ref{shelah}, there is a type $p$ over $\emptyset$ and an associated defining schema $\Psi_p$ such that 
for all $\ma\models T$ the schema $\Psi_p$ defines a 
type $p(A)$ over $A$ such that it is the only type $q$ over $A$ such that $p\subseteq q$ and $x>a \in q$ for all $a\in  A$, and moreover, if $\ma\prec\mc$ and $a\in C$ realizes $p(A)$ in $\mc$,  then $\SH(A\cup\{a\})$ is an almost minimal end 
extension of $\ma$.

Let $\ma$ be any countable model of $T$.
We define for all $\alpha<\omega_1$ models $\ma_\alpha$ and elements $a_\alpha\in A_{\alpha+1}\setminus A_\alpha$ such that $\ma_0=\ma$, for limit $\gamma$, $\ma_\gamma=\bigcup_{\alpha<\gamma}\ma_\alpha$, and
$\ma_{\alpha+1}=\SH(A_\alpha\cup\{a_\alpha\})$, where $a_\alpha$ relaizes $p(\ma_\alpha)$ in $\ma_\alpha$. By \cite{MR0209132} we can find a limit ordinal $\alpha<\omega_1$ such that $(\alpha,<)$ and $(\alpha\cdot\omega,<)$ are $\omega\cdot\delta$-equivalent. We claim  that $\ma_\alpha$ and $\ma_{\alpha\cdot\omega}$ are $\delta$-equivalent. To prove this, it is helpful to first show
that:
\medskip

\noindent{\bf CLAIM A:} If $\beta_1<\ldots<\beta_{n}<\omega_1$ and $\gamma_1<\ldots<\gamma_n<\omega_1$, then the conditions $f\restriction A=id, f(a_{\beta_i})=
a_{\gamma_i}$ determine a unique isomorphism $$\SH(A\cup\{a_{\beta_1},\ldots, a_{\beta_n}\})\to \SH(A\cup \{a_{\gamma_1},\ldots, a_{\gamma_n}\}).$$ 

We prove this by induction on $n$. Suppose the claim holds for sequences shorter than $n$. Suppose $\beta_1<\ldots<\beta_{n}<\omega_1$ and $\gamma_1<\ldots<\gamma_n<\omega_1$. By induction hypothesis, the conditions $f\restriction A=id, f(a_{\beta_i})=
a_{\gamma_i}$, $1\le i<n$, determine a unique isomorphism 
\begin{equation}\label{isom}
f:\SH(A\cup\{a_{\beta_1},\ldots, a_{\beta_{n-1}}\})\to \SH(A\cup \{a_{\gamma_1},\ldots, a_{\gamma_{n-1}}\}).
\end{equation} 
Suppose now $a\in \SH(A\cup\{a_{\beta_1},\ldots, a_{\beta_{n}}\})\setminus \SH(A\cup\{a_{\beta_1},\ldots, a_{\beta_{n-1}}\})$, e.g. $a=t(\vec{c},a_{\beta_1},\ldots, a_{\beta_{n}})$, where $\vec{c}\in A$. We define
$$f(a)=t(\vec{c},a_{\gamma_1},\ldots, a_{\gamma_{n}}).$$ To see that this is legitimate we now prove:
\medskip

\noindent{\bf CLAIM B:}
 The following two conditions are equivalent for all literals (i.e. atomic or negated atomic formulas) $\phi(\vec{x},x_1,\ldots,x_n)$:

\begin{description}
\item [(I)] $\SH(A\cup\{a_{\beta_1},\ldots, a_{\beta_n}\})\models\phi(\vec{c},a_{\beta_1},\ldots, a_{\beta_n})$
\item [(II)] $\SH(A\cup\{a_{\gamma_1},\ldots, a_{\gamma_n}\})\models\phi(\vec{c},a_{\gamma_1},\ldots, a_{\gamma_n})$

\end{description}

\medskip

It suffices to prove that (I) implies (II).
Assume (I). Then $a_{\beta_n}$ satisfies $\phi(\vec{c},a_{\beta_1},\ldots, a_{\beta_{n-1}},x)$ in $\ma_{\beta_{n}}$. Hence $\phi(\vec{c},a_{\beta_1},\ldots, a_{\beta_{n-1}},x)\in p(\ma_{\beta_{n-1}})$. Now $\psi(\vec{c},a_{\beta_1},\ldots, a_{\beta_{n-1}})$ holds in $\ma_{\beta_{n-1}}$, hence in $\SH(A\cup\{a_{\beta_1},\ldots, a_{\beta_n}\})$. By (\ref{isom}), $\psi(\vec{c},a_{\gamma_1},\ldots, a_{\gamma_{n-1}})$ holds in  $\SH(A\cup\{a_{\gamma_1},\ldots, a_{\gamma_n}\})$, hence in $\ma_{\gamma_{n-1}}$. Hence  $\phi(\vec{c},a_{\gamma_1},\ldots, a_{\gamma_{n-1}},x)\in p(\ma_{\gamma_{n-1}})$. Therefore the element $a_{\gamma_n}$ satisfies  $\phi(\vec{c},a_{\gamma_1},\ldots, a_{\gamma_{n-1}},x)$ in $\ma_{\gamma_{n}}$, and hence in $\SH(A\cup\{a_{\gamma_1},\ldots, a_{\gamma_n}\})$. We have proved (II). This finishes the proof of CLAIM B and thereby the proof of CLAIM A.

The proof that $f$ is an isomorphism is a consequence of  the above once we observe that $f$ is one-one by the above, and onto by construction. The uniqueness of $f$ is obvious.

We now prove $\ma_\alpha\equiv_\delta\ma_{\alpha\cdot\omega}$. Since $(\alpha,<)$ and $(\alpha\cdot\omega,<)$ are $\omega\cdot\delta$-equivalent, Proposition~\ref{baf} gives  a back-and-forth sequence $(J_\xi)_{\xi\le \omega\cdot\delta}$ for $((\alpha,<),(\alpha\cdot\omega,<))$. We define a sequence 
$(I_\xi)_{\xi\le\delta}$ of partial mappings $A_\alpha\to A_{\alpha\cdot\omega}$ as follows:
If $\pi\in J_{\omega\cdot\xi}$ with $$\pi=\{(a_{\beta_1},a_{\gamma_1}),\ldots, (a_{\beta_n},a_{\gamma_n})\},$$ where $a_{\beta_1}<\ldots< a_{\beta_{n}}$ and $a_{\gamma_1}<\ldots< a_{\gamma_{n}}$, then $f_\pi\in I_\xi$ is the isomorphism obtained by CLAIM A. We show: 

\medskip

\noindent{\bf CLAIM C:} $(I_\xi)_{\xi\le\delta}$ is a back-and-forth sequence for $(\ma,\mb)$.
\medskip

Suppose $f_\pi\in I_{\xi+1}$, where $\pi\in J_{\omega\cdot(\xi+1)}$ and $a\in A_\alpha$. Since $\alpha$ is a limit ordinal, $a\in A_{\beta}$ for some $\beta<\alpha$. Hence $a\in\SH(A \cup\{a_{\beta'_1},\ldots, a_{\beta'_{n'}}\})$ for some $\beta'_1<\ldots<\beta'_{n'}<\alpha$. Since $\pi\in J_{\omega\cdot\xi+n'}$, there is $\pi'\in J_{\omega\cdot\xi}$ such that ${\beta'_1},\ldots, {\beta'_{n'}}\in\dom(\pi')$. Now $a_{\beta'_1},\ldots, a_{\beta'_{n'}}\in\dom(f_{\pi'})$. Hence $a\in\dom(f_{\pi'})$. The argument is the same if we start with $a\in A_{\alpha\cdot\omega}$. CLAIM C is proved.

We are left with showing that $\ma_\alpha\not\cong\ma_{\alpha\cdot\omega}$. Suppose $\pi:\ma_\alpha\cong\ma_{\alpha\cdot\omega}$. Choose $\xi<\alpha\cdot\omega$ such that 
$\pi(A)\subseteq A_\xi$. Such a $\xi$ exists because otherwise $\pi(A)$ is cofinal in $A_{\alpha\cdot\omega}$ which is not possible because $A$ is not cofinal in $\ma_\alpha$. We now argue that $\pi(A_\alpha)\subseteq A_{\xi+\alpha}$, where
$\xi+\alpha<\alpha\cdot\omega$, a contradiction with $ran(\pi)=A_{\alpha\cdot\omega}$. 
To this end we prove by induction on $\beta\le \alpha$ that $\pi(A_\beta)\subseteq A_{\xi+\beta}$. But this follows immediately from the fact that  $\ma_{\beta+1}$ is an almost minimal extension of $\ma_\beta$, whence $\pi(\ma_{\beta+1})$ is an almost minimal extension of $\pi(\ma_\beta)$.
\end{proof}

Note that the theorem cannot hold for $\delta=\omega_1$ because for countable models $\omega_1$-equivalence implies isomorphism by a simple cardinality argument.

\def\SH{\mathop{\rm Sh}}

The \emph{Scott Height} $\SH(\ma)$ of a countable model $\ma$ is the smallest $\alpha$ such that for any $a_1,\ldots,a_n,
b_1,\ldots,b_n\in A$ if $(\ma,a_1,\ldots,a_n)\equiv_\alpha(\ma,b_1,\ldots,b_n)$, then $(\ma,a_1,\ldots,a_n)\equiv_{\alpha+1}(\ma,b_1,\ldots,b_n)$. If $\ma$ and $\mb$ are countable and $\ma\equiv_{\SH(\ma)+\omega}\mb$, then $\ma\cong\mb$ \cite{MR0342370}. By \cite{MR0200133} every countable model in a countable vocabulary has a countable Scott Height.

\begin{corollary}
The Scott Heights of countable models of $\PA$ are unbounded in $\omega_1$. 
\end{corollary}

\begin{proof}
Suppose $\alpha<\omega_1$. By Theorem~\ref{main} there are countable non-isomorphic models $\ma$ and $\mb$ of $\PA$ such that $\ma\equiv_{\alpha+\omega}\mb$. It follows that the Scott Heights of $\ma$ and $\mb$ have to be greater than $\alpha$.
\end{proof}
Let us now turn to set theory:

\begin{theorem}\label{zfc}Suppose $\ZFC$ is consistent.
Then for every $\delta<\omega_1$ there are countable models $\ma$ and $\mb$ of $\ZFC$ such that $\ma$ and $\mb$ are $\delta$-equivalent but not isomorphic.
\end{theorem} 

\begin{proof}
Suppose $\ZFC$ has a model $\mm=(M,E^\mm)$. W.l.o.g. $M\models V=L$. Let $\preceq_L$ be the canonical definable well-order of $L$. Let $L_0$ be the vocabulary $\{E\}$ of $\mm$ added with names for the definable Skolem-functions $f_{\phi(y,x_1,\ldots,x_n)}$, where $\phi(y,x_1,\ldots,x_n)$ is an $\{E\}$-formula, such that if $f_{\phi(y,x_1,\ldots,x_n)}(a_1,\ldots,a_n)=b$, then  
\begin{equation}\label{Skolem}
\begin{array}{l}
\mm\models\exists y\phi(y,a_1,\ldots,a_n)\to(\phi(b,a_1,\ldots,a_n)\wedge\\
\hspace{2cm}\forall y(\phi(y,a_1,\ldots,a_n)\to b\preceq_Ly)).
\end{array}
\end{equation} We assume the Skolem-functions to be closed under composition. Let $L^*$ consist of a new $n$-ary predicate symbol $R_{\phi(x_1,\ldots,x_n)}$ for every $L_0$-formula $\phi(x_1,\ldots,x_n)$, the idea being that $R_{\phi(x_1,\ldots,x_n)}$ is a name for the relation defined by $\phi(x_1,\ldots,x_n)$. Let $L^+=L_0\cup L^*$. Let $\mm^+$ be the canonical expansion of $\mm$ to the vocabulary $L^+$, i.e.
\begin{equation}\label{Morley}
\begin{array}{l}
\mm^+\models \forall y_1,\ldots,y_n (R_{\phi(x_1,\ldots,x_n)}(y_1,\ldots,y_n)\leftrightarrow\phi(y_1,\ldots,y_n))
\end{array}
\end{equation}

We let $\mm^*$ be the submodel of $\mm^+\restriction L^*$ with universe $\omega^\mm$ equipped with the functions $h_g$ for the definable Skolem functions $g\in L_0$, where $h_g$ is such that for $x_1,\ldots,x_n\in M^*$
$$h_g(x_1,\ldots,x_n)=\left\{
\begin{array}{ll}
g(x_1,\ldots,x_n)&\mbox{ if $g(x_1,\ldots,x_n)\in M^*$}\\
0&\mbox{otherwise.}
\end{array}\right.$$ These are partial definable functions in $M$ and serve as Skolem-functions in $\mm^*$. Let $T^*=\Th(\mm^*)$. We think of $T^*$ as an extension of $\PA$ since any model of $\ZFC$ satisfies $\PA$ for its arithmetic part and the functions $+,\cdot$ are clearly definable in models of $T^*$. 

Suppose  $\ma^*\models T^*$. We define a new model  $\ma$ as follows. Let $N$ be the set of tuples $\langle f_\phi,a_1,\ldots,a_n\rangle$, where   $f_\phi$ is the name of a canonical definable Skolem function in $L_0$ in the variables $x_1,\dots,x_n$ and $a_1,\ldots,a_n\in A^*$. Define in $N$:
$$\langle f_\phi,a_1,\dots,a_n\rangle\sim \langle f_\psi,a'_1,\dots,a'_{n'}\rangle\iff$$ 
\begin{equation}\label{identity}
\ma^*\models R_{f_\phi(x_1,\dots,x_n)=f_\psi(x'_1,\dots,x'_{n'})}(a_1,\ldots,a_n,a'_1,\ldots,a'_{n'}).
\end{equation} It is easy to see that this is an equivalence relation on $N$. Let $A$ be the set of equivalence classes $[\langle f_\phi,a_1,\dots,a_n\rangle]$. Define on $A$:
$$[\langle f_\phi,,a_1,\dots,a_n\rangle] E^\circ  [\langle f_\psi,a'_1,\dots,a'_{n'}\rangle]\iff$$ 
\begin{equation}\label{epsilon}
\ma^*\models R_{f_\phi(x_1,\dots,x_n)\hspace{1pt} E\hspace{1pt}  f_\psi(x'_1,\dots,x'_{n'})}(a_1,\ldots,a_n,a'_1,\ldots,a'_{n'}).
\end{equation}
For any $n$-ary Skolem-function $f_\phi$ in $L_0$ we define
$$f_\phi^\circ([\langle f^1,a^1_1,\dots,a^1_{m_1}\rangle],\ldots,[\langle f^n,a^n_1,\dots,a^n_{m_n}\rangle])=$$
$$[\langle g,a^1_1,\dots,a^1_{m_1},\ldots,a^n_1,\dots,a^n_{m_1}\rangle]),$$ where $g$ is obtained from $f$ and $f^1,\ldots,f^n$ by composition.
The properties of the  $R$-predicates in $\mm$ imply that the conditions (\ref{identity}) and (\ref{epsilon}) are independent of the choice of the representative of the relevant $\sim$-classes.
Let $\ma=(A,E^\circ,(f^\circ)_{f\in L_0})$ and $\bar{\ma}=(A,E^\circ)$. We call $\ma$ the model \emph{derived} from $\ma^*$. 

There is a canonical mapping $j_\ma:A^*\to A$ defined by $j_\ma(a)=[\langle id,a\rangle]$, where $id$ is the Skolem function $f_{y=x_1}$.
\medskip

\noindent{\bf Claim D:} For any $L_0$-formula $\phi(x_1,\ldots,x_n)$ and any $a_1,\ldots,a_n\in A^*$ the following are equivalent:
\begin{description}
\item[(i)] $\ma^*\models R_{\phi(f_1(\vec{x}),\ldots,f_n(\vec{x}))}(\vec{a})$.
\item[(ii)] $\ma\models\phi([\langle f_1,\vec{a}\rangle],\ldots,[\langle f_n,\vec{a}\rangle])$.
\end{description}\medskip

Proof of the Claim: For atomic $\phi(x_1,\ldots,x_n)$  in the vocabulary $L_0$ the claim is true by definition. The properties of the predicates $R^{\mm^+}_\phi$, $R_\phi\in L^+$, arising from (\ref{Skolem}) and (\ref{Morley}), can be used to prove the claim by induction for composite $L_0$-formulas.
\medskip

In particular, $j_\ma$ is a bijection into some $A_0\subseteq A$. 

We can easily see that  $\bar{\ma}\models\Th(\mm)$, for suppose $\phi$ is an $\{E\}$-sentence such that $\mm\models\phi$. Then $\mm^+\models R_{\phi}()$, and hence $\mm^*\models R_\phi()$. It follows that $\ma^*\models R_\phi()$. By the above Claim D, $\bar{\ma}\models\phi$.

We now show that  $\omega^\ma=A_0$, as we need this information later: Suppose $\ma\models ``b=[\langle f_{\phi},\vec{a}\rangle]\in \omega"$, where $\phi=\phi(y,\vec{x})$ is an $L_0$-formula and $\vec{a}\in A^*$. 
Let $\theta(y)$ say in the vocabulary $\{E\}$ that $y$ is a natural number. Thus $\ma\models\theta(b)$. Now by Claim D,
\begin{equation}\label{123}
\ma^*\models 
R_{\theta(f_\phi(\vec{x}))}(\vec{a}).
\end{equation} %

Let $\phi'=\phi'(y,\vec{x})$ say  in the vocabulary $\{E\}$ that there are $y$ such that 
$\phi(y,\vec{x})$ and $y$ is the $\preceq_L$-least of them.   
 
Note that
$$\mm^*\models\forall\vec{z}((R_{\phi'(f_{\phi}(\vec{x}),\vec{x})}(\vec{z})\wedge 
R_{\theta(f_{\phi}(\vec{x}))}(\vec{z})))\to
 R_{\phi'}(h_{f_{\phi}}(\vec{z}),\vec{z})),$$ whence
$$\ma^*\models (R_{\phi'(f_{\phi}(\vec{x}),\vec{x})}(\vec{a})\wedge 
R_{\theta(f_{\phi}(\vec{x}))}(\vec{a}))\to
 R_{\phi'}(h_{f_{\phi}}(\vec{a}),\vec{a})).$$
 From this and (\ref{123}) we obtain
 $\ma^*\models 
 R_{\phi'}(h_{f_{\phi}}(\vec{a}),\vec{a})).$ Let $c=h_{f_{\phi}}(\vec{a})$. 
Note again that
$$\mm^*\models\forall\vec{z}\forall u\forall u'((R_{\phi'}(u,\vec{z})\wedge 
R_{\phi'}(u',\vec{z}))\to
u=u'),$$ whence
$j_\ma(c)=b$ and $b\in A_0$ follows. We have proved $\omega^\ma=A_0$.

The theory $T^*$ clearly satisfies (A1)-(A3).  
By Theorem~\ref{main} there are countable non-isomorphic $\ma^*,\mb^*$ such that both are models of $T^*$ and in addition $\ma^*\equiv_{\omega\cdot\delta}\mb^*$. They give rise to their derived models $\ma$ and $\mb$, as well as to bijections $j_\ma$ and $j_\mb$ onto   $A_0\subseteq A$ and $B_0\subseteq B$, respectively, as above. By the above, $\bar{\ma}$ and $\bar{\mb}$ are countable models of $\ZFC$.

Let us note that $\bar{\ma}\equiv_\delta\bar{\mb}$, for  there is a back-and-forth sequence $(P^*_\beta)_{\beta\le\omega\cdot\delta}$ for $\ma^*$ and $\mb^*$, and it is easy to derive from this a back-and-forth sequence $(P_\beta)_{\beta\le\delta}$ for $\bar{\ma}$ and $\bar{\mb}$. On the other hand,  we can easily conclude that $\bar{\ma}\not\cong\bar{\mb}$. Indeed, suppose $\sigma:\bar{\ma}\cong\bar{\mb}$. Because the Skolem functions $f_\phi$ are definable in $\bar{\ma}$ and $\bar{\mb}$, we have $\sigma:{\ma}\cong{\mb}$. If $a\in A^*$, then $j_\ma(a)\in A_0=\omega^\ma$, whence $\sigma(j_\ma(a))\in \omega^\mb$. There is $b\in B^*$ such that $\sigma(j_\ma(a))=j_\mb(b)$. We let $\sigma^*(a)=b$. It is clear that this is a well-defined bijection $A^*\to B^*$. Now clearly,  denoting the interpretation of $f_\phi$ in $\ma$ by $f^A_\phi$ and in $\mb$ by $f^B_\phi$,
$$\begin{array}{lcl}
\sigma([\langle f_\phi,a_1,\ldots,a_n\rangle]&=&
\sigma(f^A_\phi([\langle id,a_1\rangle],\ldots,[\langle id,a_n\rangle]))\\
&=&f^B_\phi(\sigma([\langle id,a_1\rangle]),\ldots,\sigma([\langle id,a_n\rangle])))\\
&=&[\langle f_\phi,\sigma^*(a_1),\ldots,\sigma^*(a_n)\rangle]\\
\end{array}$$
for all Skolem-functions $f_\phi$ in $L_0$ and all $a_1,\ldots,a_n\in A^*$.
 Moreover,
{\setlength{\arraycolsep}{2pt}$$
\begin{array}{lcl}
\ma^*&\models& R_{\phi(f_1(\vec{x}),\ldots,f_n(\vec{x}))}(\vec{a})\iff\\
\ma&\models&\phi([\langle f_1,\vec{a}\rangle],\ldots,[\langle f_n,\vec{a}\rangle])\iff\\
\mb&\models&\phi(\sigma([\langle f_1,\vec{a}\rangle]),\ldots,\sigma([\langle f_n,\vec{a}\rangle]))\iff\\
\mb&\models&\phi([\langle f_1,\sigma^*(\vec{a})\rangle],\ldots,[\langle f_n,\sigma^*(\vec{a})\rangle])\iff\\
\mb^*&\models& R_{\phi(f_1(\vec{x}),\ldots,f_n(\vec{x}))}(\sigma^*(\vec{a})).
\end{array}$$}
Similarly $\sigma^*$ preserves the functions $h_g$. In the end, $\sigma^*:\ma^*\cong\mb^*$, contrary to our assumption.

\end{proof}

\begin{corollary}
The Scott Heights of countable models of $\ZFC$ are unbounded in $\omega_1$. 
\end{corollary}


\end{document}